\documentclass[11pt]{amsart}
\usepackage[T1]{fontenc}
\usepackage{lmodern}
\usepackage{amsmath,amssymb,mathtools}
\usepackage{microtype}
\usepackage{enumitem}
\usepackage[hidelinks]{hyperref}
\hypersetup{
  pdftitle={Second-order fusion asymptotics for Sine beta correlation functions},
  pdfauthor={Weiyang Fang}
}
\usepackage[margin=1.0in]{geometry}

\newtheorem{theorem}{Theorem}[section]
\newtheorem{proposition}[theorem]{Proposition}
\newtheorem{lemma}[theorem]{Lemma}
\newtheorem{corollary}[theorem]{Corollary}
\newtheorem{conjecture}[theorem]{Conjecture}
\theoremstyle{remark}
\newtheorem{remark}[theorem]{Remark}

\newcommand{\Sine}{\mathsf{Sine}}
\newcommand{\HP}{\mathsf{HP}}
\newcommand{\CJ}{\mathsf{CJ}}
\newcommand{\E}{\mathbb E}
\newcommand{\R}{\mathbb R}
\newcommand{\PV}{\operatorname{PV}}

\title[Second-order fusion asymptotics]{Second-order fusion asymptotics for $\Sine_\beta$ correlation functions}
\author{Weiyang Fang}
\address{Independent Researcher, Tokyo, Japan}
\email{weiyang.fang@hotmail.com}
\date{August 24, 2026}

\begin{document}
\begin{abstract}
Recent work gives an all-$\beta$, all-order stochastic-zeta representation of the correlation functions of the $\Sine_\beta$ process and determines their leading Vandermonde asymptotics when several variables merge. We compute the first nontrivial correction throughout the regime $m\beta>1$. If $a_1,\ldots,a_m$ are distinct real numbers and
\[
V(a)=\sum_{1\le i<j\le m}(a_i-a_j)^2,
\]
then, as $\varepsilon\to0$,
\[
\rho^{(m)}_\beta(\varepsilon a_1,\ldots,\varepsilon a_m)
=C^{(m)}_\beta |\varepsilon|^{\beta\binom m2}\prod_{i<j}|a_i-a_j|^\beta
\left[1-\frac{\beta^2 V(a)}{8(m\beta-1)(2m+1)}\varepsilon^2+o(\varepsilon^2)\right].
\]
For $m=2$ the normalized second-order coefficient is $-\beta^2/[40(2\beta-1)]$ for every $\beta>1/2$. The proof combines a finite-$N$ rotational Ward identity, exact Hua--Pickrell trace moments, compact moment bounds for the stochastic-zeta entire function and its derivatives, and a quantitative multivariate expectation--Taylor lemma. As a by-product we evaluate
\[
\E_{\HP_{\beta,m\beta/2}}\sum_x x^{-2}=\frac{m\beta}{4(m\beta-1)(2m+1)}.
\]
The pole at $m\beta=1$ marks the boundary of the present second-moment argument and suggests a transition in the form of the next fusion correction.
\end{abstract}
\maketitle

\section{Introduction}
The $\Sine_\beta$ process is the translation-invariant bulk limit of beta ensembles, with intensity $1/(2\pi)$ in the normalization used below. It admits several complementary descriptions, including the Brownian carousel and random Dirac operator formulations of Valk\'o and Vir\'ag \cite{VV2009,VV2017}, and the circular-ensemble operator limits of Li and Valk\'o \cite{LiValko2022}. Until recently, explicit information on correlation functions for general $\beta$ was scarce outside the classical values $\beta=1,2,4$.

Qu and Valk\'o \cite{QuValko2025} obtained a stochastic differential equation representation of the two-point function for every $\beta>0$, as well as power-series expansions for even integer $\beta$. More recently, Qiu, Qu, Yuan, Valk\'o, and Venancio derived an explicit special-function representation for the $\Sine_6$ pair correlation \cite{QiuEtAl2026}. Assiotis and Najnudel \cite{AssiotisNajnudel2026} subsequently proved the first all-order, all-$\beta$ formula. In their normalization, for $m\in\mathbb N$ and $\beta>0$,
\begin{equation}\label{eq:AN}
\rho^{(m)}_\beta(x_1,\ldots,x_m)
=C^{(m)}_\beta\prod_{i<j}|x_i-x_j|^\beta
\E\prod_{j=2}^m\bigl|\xi^{\beta,m\beta/2}(x_j-x_1)\bigr|^\beta,
\end{equation}
where $\xi^{\beta,\delta}$ is the Hua--Pickrell stochastic zeta function. They deduced the leading fusion law
\[
\rho^{(m)}_\beta(x_1,\ldots,x_m)\sim C^{(m)}_\beta\prod_{i<j}|x_i-x_j|^\beta
\]
as the variables merge and explicitly singled out the lower-order terms, already for $m=2$, as an open direction \cite[Corollary 1.9 and the discussion preceding it]{AssiotisNajnudel2026}.

The purpose of this paper is to compute that first lower-order term throughout the regime $m\beta>1$. This is exactly the range in which the Hua--Pickrell trace second moment entering the Ward identity is finite and in which the second-order argument below closes. The result is genuinely all-$\beta$ and all-order in the number $m$ of merging variables: it is not obtained by interpolating the classical values $\beta=1,2,4$ or the even-integer expansions available for the pair correlation. The collision geometry enters only through the quadratic discriminant
\begin{equation}\label{eq:V}
V(a):=\sum_{1\le i<j\le m}(a_i-a_j)^2
=m\sum_{j=1}^m(a_j-\bar a)^2,
\qquad \bar a:=\frac1m\sum_{j=1}^m a_j.
\end{equation}

\begin{theorem}[Second-order fusion asymptotics]\label{thm:main}
Let $\beta>0$, $m\ge2$, and suppose $m\beta>1$. For pairwise distinct $a_1,\ldots,a_m\in\R$, as $\varepsilon\to0$,
\begin{align}\label{eq:main}
\rho^{(m)}_\beta(\varepsilon a_1,\ldots,\varepsilon a_m)
&=C^{(m)}_\beta|\varepsilon|^{\beta\binom m2}\prod_{i<j}|a_i-a_j|^\beta\notag\\
&\quad\times\left[1-\frac{\beta^2}{8(m\beta-1)(2m+1)}V(a)\varepsilon^2+o(\varepsilon^2)\right].
\end{align}
The remainder is locally uniform when $a$ ranges over compact subsets of the set of pairwise distinct configurations.
\end{theorem}

\begin{corollary}[Pair fusion]\label{cor:pair}
For every $\beta>1/2$,
\begin{equation}\label{eq:pair}
\rho^{(2)}_\beta(0,x)=C^{(2)}_\beta|x|^\beta\left[1-\frac{\beta^2}{40(2\beta-1)}x^2+o(x^2)\right],\qquad x\to0.
\end{equation}
\end{corollary}

Two features of Theorem~\ref{thm:main} are worth emphasizing. First, the shape dependence of the first correction collapses to the single quadratic discriminant $V(a)$. Second, the coefficient has a simple pole at $m\beta=1$. The proof shows that the same boundary is where the trace second moment used by the Ward identity ceases to be finite; Section~\ref{sec:threshold} explains the corresponding inverse-square heuristic without using it in the proof of the theorem.

The theorem supplies an explicit next-order term in the fusion problem left open after the leading asymptotics of \cite{AssiotisNajnudel2026}. For $m=2$, it extends beyond the classical values and the even-integer power-series calculations of \cite{QuValko2025}; for $m\ge3$ it gives a uniform formula for arbitrary $\beta$ in the full supercritical region $m\beta>1$. In the determinantal case $\beta=2$, the $m=3$ and $m=4$ specializations are independently recovered from direct sine-kernel determinant expansions in Section~\ref{sec:checks}.

\begin{theorem}[An inverse-square Hua--Pickrell moment at fusion]\label{thm:inverse}
Let $\beta>0$, $m\ge2$, and $m\beta>1$. Let $X_{\beta,m\beta/2}$ denote the zero set of $\xi^{\beta,m\beta/2}$, equivalently the $\HP_{\beta,m\beta/2}$ point process. Then
\begin{equation}\label{eq:inverse}
\E\sum_{x\in X_{\beta,m\beta/2}}\frac1{x^2}
=\frac{m\beta}{4(m\beta-1)(2m+1)}.
\end{equation}
\end{theorem}

For $\beta=2$, $m=2$, \eqref{eq:inverse} gives $1/15$ in the present mean-spacing-$2\pi$ normalization. After unfolding to mean spacing one this becomes $4\pi^2/15$, agreeing with the direct determinantal computation for the confluent two-Palm sine-kernel process.

The main algebraic input is elementary: a rotational Ward identity at finite $N$. What makes the identity useful is that the slope of the normalized characteristic polynomial is exactly the trace variable whose limiting moments were computed by Assiotis--Gunes--Soor \cite{AGS2022}. The analytic input is different from a near-zero Palm truncation: we use the complex-plane moment estimate established in the proof of Assiotis--Najnudel \cite[Proposition 2.10, equations (37)--(38)]{AssiotisNajnudel2026}, their sharp real-axis moment bound \cite[Theorem 2.11]{AssiotisNajnudel2026}, Cauchy estimates for derivatives, and a quantitative expectation--Taylor lemma. This gives the required \emph{joint} second-order expansion throughout $m\beta>1$.

The same threshold suggests a broader phase diagram. We record it only as a conjecture, since the critical and subcritical localization estimates require additional work.
\begin{conjecture}[Fusion transition at $m\beta=1$]\label{conj:transition}
Let $R_{\beta,m}(\varepsilon;a)$ denote the bracketed normalized fusion factor in \eqref{eq:main}. Then the first correction changes type at $m\beta=1$:
\[
R_{\beta,m}(\varepsilon;a)-1\asymp
\begin{cases}
|\varepsilon|^{m\beta+1},&m\beta<1,\\
\varepsilon^2\log(1/|\varepsilon|),&m\beta=1,\\
\varepsilon^2,&m\beta>1.
\end{cases}
\]
The supercritical coefficient is the one in Theorem~\ref{thm:main}.
\end{conjecture}

The heuristic behind Conjecture~\ref{conj:transition} is described in Section~\ref{sec:threshold}. It is an inverse-square integrability transition in the fused Palm environment, rather than a change in the leading Vandermonde repulsion.

\paragraph{Organization.}
Section~\ref{sec:cj} recalls circular Jacobi ensembles and stochastic zeta. Section~\ref{sec:ward} proves the finite-$N$ Ward identity. Section~\ref{sec:slope} imports the Hua--Pickrell trace moments and records compact derivative bounds. Section~\ref{sec:localization} establishes the joint second-order stochastic-zeta expansion. Section~\ref{sec:proof} proves Theorems~\ref{thm:main} and \ref{thm:inverse}. Classical checks at $\beta=2,4$ are given in Section~\ref{sec:checks}. Section~\ref{sec:threshold} discusses the critical threshold and open problems.

\section{Circular Jacobi ensembles and stochastic zeta}\label{sec:cj}
We use the normalization of \cite{AssiotisNajnudel2026,QuValko2025}. Let $\theta_1,\ldots,\theta_N\in[-\pi,\pi)$ have the circular beta ensemble law
\[
\frac1{Z_{N,\beta}}\prod_{1\le j<k\le N}|e^{i\theta_j}-e^{i\theta_k}|^\beta\prod_{j=1}^N\frac{d\theta_j}{2\pi}.
\]
For $\delta>-1/2$, the circular Jacobi ensemble $\CJ_{N,\beta,\delta}$ is obtained by the additional weight
\[
\prod_{j=1}^N|1-e^{i\theta_j}|^{2\delta}.
\]
The Cayley variables
\[
x_j=\cot(\theta_j/2)
\]
produce the finite Hua--Pickrell beta ensemble; see \cite{AGS2022,LiValko2022}.

Let
\[
X_N(z):=\prod_{j=1}^N(z-e^{i\theta_j}),\qquad z\in\mathbb C,
\]
and define, for real $t$ in a neighborhood of zero,
\begin{equation}\label{eq:rN}
r_N(t):=\frac{|X_N(e^{it/N})|}{|X_N(1)|}.
\end{equation}
Almost surely $r_N$ is smooth near zero and $r_N(0)=1$. Direct differentiation gives
\begin{align}
\frac{r_N'(0)}{r_N(0)}&=-\frac1{2N}\sum_{j=1}^N\cot(\theta_j/2),\label{eq:rprime}\\
\frac{r_N''(0)}{r_N(0)}&=r_N'(0)^2-\frac1{4N^2}\sum_{j=1}^N\csc^2(\theta_j/2).\label{eq:rsecond}
\end{align}
Equation~\eqref{eq:rprime} is the microscopic version of the characteristic-polynomial derivative/trace identity used in \cite[Proposition 2.2]{AGS2022}.

Li and Valk\'o \cite{LiValko2022} proved an operator-level limit of $\CJ_{N,\beta,\delta}$ and convergence of the normalized characteristic polynomial to a random entire function. Assiotis and Najnudel identify this limit with the principal-value stochastic zeta function
\begin{equation}\label{eq:xi}
\xi^{\beta,\delta}(z)=\PV\prod_{x\in X_{\beta,\delta}}\left(1-\frac zx\right),
\qquad \xi^{\beta,\delta}(0)=1,
\end{equation}
whose zero set is $\HP_{\beta,\delta}$; see \cite[Propositions 2.7--2.8]{AssiotisNajnudel2026}. Under a suitable coupling the normalized polynomials converge locally uniformly on $\mathbb C$, hence so do all complex derivatives.

For later use set
\begin{equation}\label{eq:h}
h_{\beta,\delta}(t):=\E|\xi^{\beta,\delta}(t)|^\beta,
\qquad t\in\R.
\end{equation}
The sharp moment estimate of \cite[Theorem 2.11 and Corollary 2.12]{AssiotisNajnudel2026} provides the local joint moment control used below. Products of powers whose total exponent is at most $2\delta$ are uniformly bounded in expectation; when the total exponent is strictly below $2\delta$, the exponents may be increased slightly, which gives uniform integrability on compact sets.

\section{A finite-\texorpdfstring{$N$}{N} rotational Ward identity}\label{sec:ward}
The following identity is the algebraic core of the argument.  We state it only in the range in which the differentiated quantities are integrable.
\begin{proposition}[Rotational Ward identity]\label{prop:ward}
Let $\delta>1/2$, $\beta>0$, and let expectation $\E_{N,\delta}$ be with respect to $\CJ_{N,\beta,\delta}$. With $r_N$ as in \eqref{eq:rN},
\begin{equation}\label{eq:ward}
\E_{N,\delta}\left[r_N''(0)+(2\delta-1)r_N'(0)^2\right]=0.
\end{equation}
Consequently, if $h_N(t):=\E_{N,\delta}r_N(t)^\beta$, then
\begin{equation}\label{eq:hNsecond}
h_N''(0)=\beta(\beta-2\delta)\E_{N,\delta}r_N'(0)^2.
\end{equation}
\end{proposition}
\begin{proof}
By rotational invariance of the underlying circular beta ensemble,
\[
u\longmapsto \E_{\mathrm{C}\beta\mathrm E_N}|X_N(e^{iu})|^{2\delta}
\]
is constant. Dividing by its value at $u=0$, changing measure to $\CJ_{N,\beta,\delta}$, and writing $t=Nu$ gives the exact identity
\begin{equation}\label{eq:rotation-exact}
\E_{N,\delta}r_N(t)^{2\delta}=1
\end{equation}
for every real $t$.

For completeness, the differentiation in \eqref{eq:rotation-exact} can be justified without a pointwise envelope for the moving singularity. Put
\[
 g(\theta):=|1-e^{i\theta}|^{2\delta},\qquad \theta\in\mathbb T.
\]
When $\delta>1/2$, one has $g\in W^{2,1}(\mathbb T)$: near the origin, $g''(\theta)=O(|\theta|^{2\delta-2})$, which is integrable. Hence the translation map $u\mapsto g(\cdot-u)$ is $C^2$ as an $L^1(\mathbb T)$-valued map. After the change of measure, the numerator of \eqref{eq:rotation-exact} is, up to its constant normalizing factor,
\[
 \int_{\mathbb T^N}\prod_{j<k}|e^{i\theta_j}-e^{i\theta_k}|^\beta
 \prod_{j=1}^N g(\theta_j-t/N)\,d\boldsymbol\theta.
\]
The Vandermonde factor is bounded on $\mathbb T^N$, while the first two $L^1$ derivatives of the translated product are finite sums of terms containing either one copy of $g''$ or two copies of $g'$, with all remaining factors bounded. Thus the displayed integral is twice differentiable in $t$ and its derivatives are obtained by differentiating under the integral. Differentiating \eqref{eq:rotation-exact} twice at $t=0$ therefore yields
\[
0=2\delta\,\E_{N,\delta}\left[r_N''(0)+(2\delta-1)r_N'(0)^2\right],
\]
which proves \eqref{eq:ward}. Since $r_N(0)=1$,
\[
h_N''(0)=\beta\E r_N''(0)+\beta(\beta-1)\E r_N'(0)^2.
\]
Substitution of \eqref{eq:ward} gives \eqref{eq:hNsecond}.
\end{proof}
\begin{remark}
No large-$N$ asymptotics are used in Proposition~\ref{prop:ward}.  The restriction $\delta>1/2$ is imposed only to justify the twice-differentiated expectation; this is exactly the range used later.
\end{remark}

\section{Trace moments and compact analytic bounds}\label{sec:slope}
We first identify the slope.  Let
\[
\mathfrak q_N(z):=\frac{X_N(z)}{X_N(1)},
\qquad
f_N(z):=e^{-iz/2}\mathfrak q_N(e^{iz/N}).
\]
This is the normalization used in \cite[Definition 2.2 and Proposition 2.7]{AssiotisNajnudel2026}. For real $t$, $f_N(t)$ is real and $f_N(0)=1$; hence, until the first zero is crossed, $f_N(t)=r_N(t)$. In particular,
\begin{equation}\label{eq:f-r-derivatives}
f_N'(0)=r_N'(0),\qquad f_N''(0)=r_N''(0).
\end{equation}
Assiotis--Najnudel construct a coupling in which $f_N\to\xi^{\beta,\delta}$ almost surely, locally uniformly on $\mathbb C$ \cite[Propositions 2.7--2.8]{AssiotisNajnudel2026}; therefore every fixed complex derivative converges almost surely as well.

Assiotis--Gunes--Soor \cite{AGS2022} study the Hua--Pickrell trace variable $X_\beta(\tau)$ obtained as the limit of the normalized sum of Cayley coordinates. Their explicit second moment is
\begin{equation}\label{eq:trace-second}
\E X_\beta(\tau)^2=\frac{\beta}{(2\tau-1)(4\tau+\beta)},
\qquad \tau>\frac12.
\end{equation}
More generally, Proposition 3.8 of \cite{AGS2022} gives convergence of absolute moments of every real order $0\le t\le r$ once the one-row variable has an $r$th moment, and Proposition 3.11 specializes this to the Hua--Pickrell trace: for every real $h$ with $0\le h<\Re(\tau)+1/2$, the absolute moment of order $p=2h$ is finite and the corresponding finite-$N$ moments converge. Thus all real orders $0\le p<2\Re(\tau)+1$ are covered; no integrality of $p$ is required.

\begin{proposition}[Slope moments]\label{prop:slope}
Let $\delta>1/2$. Then, for every $0<p<2\delta+1$,
\begin{equation}\label{eq:slope-p}
\E|\xi^{\beta,\delta\,\prime}(0)|^p<\infty,
\end{equation}
and for $p=2$,
\begin{equation}\label{eq:slope-moment}
\lim_{N\to\infty}\E_{N,\delta}r_N'(0)^2
=\E\bigl[\xi^{\beta,\delta\,\prime}(0)^2\bigr]
=\frac{\beta}{4(2\delta-1)(4\delta+\beta)}.
\end{equation}
\end{proposition}
\begin{proof}
By \eqref{eq:rprime} and \eqref{eq:f-r-derivatives},
\[
f_N'(0)=r_N'(0)=-\frac1{2N}\sum_{j=1}^N x_j,
\qquad x_j=\cot(\theta_j/2).
\]
Under $\CJ_{N,\beta,\delta}$, $(x_1,\ldots,x_N)$ has the finite Hua--Pickrell law with parameter $\delta$. By \cite[Proposition 3.8]{AGS2022}, moment convergence holds for every real exponent up to any available one-row moment; its Hua--Pickrell specialization \cite[Proposition 3.11]{AGS2022}, with $p=2h<2\delta+1$, gives exactly the asserted real (not necessarily integer) moment range and convergence of the normalized-trace moments. On the other hand, the coupling of \cite[Proposition 2.7]{AssiotisNajnudel2026} gives $f_N'(0)\to\xi'(0)$ almost surely. Thus the limiting derivative has the same law as $-X_\beta(\delta)/2$, which proves \eqref{eq:slope-p}. For $p=2$, substituting \eqref{eq:trace-second} gives \eqref{eq:slope-moment}. Notice that no claim that bounded second moments alone imply uniform integrability is needed here; the identification uses the already established trace limit and its moment convergence.
\end{proof}

The next two lemmas isolate the precise complex-plane input and then convert it into derivative control.
\begin{lemma}[Complex-plane pointwise moments]\label{lem:complex-pointwise}
Let $\beta,\delta>0$ and $0<s\le2\delta$. There is a constant $C=C(\beta,\delta,s)$ such that, for every $N\ge1$ and every $z=x+iy\in\mathbb C$,
\begin{equation}\label{eq:complex-global}
\E|f_N(z)|^s\le C e^{s|y|/2}.
\end{equation}
Moreover,
\begin{equation}\label{eq:complex-limit-global}
\E|\xi^{\beta,\delta}(z)|^s\le C e^{s|z|/2}.
\end{equation}
In particular, both expectations are uniformly bounded when $z$ ranges over a fixed compact set.
\end{lemma}
\begin{proof}
This is the complex-plane estimate proved in the course of \cite[Proposition 2.10]{AssiotisNajnudel2026}. For $z=x+iy$, equation (37) there relates the limiting entire function to the normalized circular-Jacobi polynomial at $e^{-y/N}e^{ix/N}$. Equation (38), using the Poisson kernel and \cite[Theorem 2.11]{AssiotisNajnudel2026}, bounds the latter uniformly when the radius is at most one; the functional equation displayed immediately after (38) gives the corresponding bound outside the unit disk. Since $|e^{-iz/2}|=e^{y/2}$, these estimates give \eqref{eq:complex-global}. Passing to the coupled limit by Fatou's lemma gives \eqref{eq:complex-limit-global}; the proof of \cite[Proposition 2.10]{AssiotisNajnudel2026} records the latter explicitly in the form $\E|\xi(z)|^s\le C e^{s|z|/2}$.
\end{proof}

\begin{lemma}[Compact moments of derivatives]\label{lem:analytic-moments}
Let $\beta,\delta>0$, let $0<s<2\delta$, let $R>0$, and let $k\ge0$ be an integer. Then
\begin{equation}\label{eq:derivative-moment-bound}
\sup_{N\ge1}\E\sup_{|z|\le R}|f_N^{(k)}(z)|^s<\infty,
\qquad
\E\sup_{|z|\le R}|\xi^{\beta,\delta\,(k)}(z)|^s<\infty.
\end{equation}
Consequently, for $s>1$ the families $\{f_N^{(k)}(0)\}_N$ are uniformly integrable.
\end{lemma}
\begin{proof}
For an entire function $g$, $|g|^s$ is subharmonic. Applying the sub-mean/Poisson inequality on a disk of radius $R+1$ bounds $\sup_{|z|\le R}|g(z)|^s$ by a constant times the boundary integral of $|g|^s$ on $|z|=R+1$. Taking expectations and using Lemma~\ref{lem:complex-pointwise} gives the case $k=0$ of \eqref{eq:derivative-moment-bound}, uniformly in $N$ and also for $\xi$. Cauchy's estimate on the concentric disks $R$ and $R+1/2$ then yields the same conclusion for every fixed derivative order $k$. The final assertion follows from the $L^s$ bound when $s>1$.
\end{proof}

\section{Joint second-order expansion of the stochastic-zeta factor}\label{sec:localization}
Let $X_{\beta,\delta}$ denote the zero set of $\xi^{\beta,\delta}$.  Proposition 2.8 of \cite{AssiotisNajnudel2026} identifies $\xi$ with its principal-value product and, through the canonical-product representation used in that proof, gives almost surely
\[
\sum_{x\in X_{\beta,\delta}}\frac1{x^2}<\infty.
\]
Thus, with
\begin{equation}\label{eq:S}
S_1:=\PV\sum_{x\in X_{\beta,\delta}}\frac1x,
\qquad
S_2:=\sum_{x\in X_{\beta,\delta}}\frac1{x^2},
\end{equation}
one has pathwise
\begin{equation}\label{eq:derivatives}
\xi'(0)=-S_1,
\qquad
\xi''(0)=S_1^2-S_2.
\end{equation}

We first determine the expectations in \eqref{eq:derivatives} without any Palm truncation.
\begin{proposition}[Inverse moments from the Ward limit]\label{prop:moments}
Let $m\ge2$, $\beta>0$, $m\beta>1$, and put $\delta=m\beta/2$. Then
\begin{align}
\E S_1^2&=\frac{1}{4(m\beta-1)(2m+1)},\label{eq:S1value}\\
\E S_2&=\frac{m\beta}{4(m\beta-1)(2m+1)}.\label{eq:S2value}
\end{align}
Moreover,
\begin{equation}\label{eq:ward-limit}
\E\xi''(0)=-(2\delta-1)\E\xi'(0)^2.
\end{equation}
\end{proposition}
\begin{proof}
The first formula follows immediately from Proposition~\ref{prop:slope}: since $2\delta=m\beta$,
\[
\E S_1^2=\E\xi'(0)^2
=\frac{\beta}{4(2\delta-1)(4\delta+\beta)}
=\frac1{4(m\beta-1)(2m+1)}.
\]
Choose $s$ with $1<s<2\delta$, which is possible because $m\beta=2\delta>1$. By Lemma~\ref{lem:analytic-moments}, $f_N''(0)\to\xi''(0)$ almost surely and $\{f_N''(0)\}$ is uniformly integrable. Therefore
\[
\E f_N''(0)\longrightarrow\E\xi''(0).
\]
Proposition~\ref{prop:slope} gives convergence of the second moments of $f_N'(0)=r_N'(0)$. Passing \eqref{eq:ward} to the limit using \eqref{eq:f-r-derivatives} proves \eqref{eq:ward-limit}.

By Lemma~\ref{lem:analytic-moments}, $\xi''(0)\in L^1$, while $S_1^2\in L^1$ by Proposition~\ref{prop:slope}. Equation \eqref{eq:derivatives} therefore implies $S_2=S_1^2-\xi''(0)\in L^1$. Taking expectations and using \eqref{eq:ward-limit} gives
\[
\E S_2
=\E S_1^2-\E\xi''(0)
=2\delta\,\E S_1^2,
\]
which is exactly \eqref{eq:S2value}.
\end{proof}

\begin{lemma}[Joint stochastic-zeta Taylor expansion]\label{lem:joint}
Let $m\ge2$, $\beta>0$, $m\beta>1$, set $\delta=m\beta/2$ and $d=m-1$, and write $\xi=\xi^{\beta,\delta}$. For $b=(b_1,\ldots,b_d)\in\R^d$ define
\[
H_b(t):=\E\prod_{j=1}^d|\xi(tb_j)|^\beta,
\qquad
A(b):=\sum_{j=1}^d b_j,
\qquad
B(b):=\sum_{j=1}^d b_j^2.
\]
Then, uniformly for $b$ in compact subsets of $\R^d$,
\begin{equation}\label{eq:joint-expansion}
H_b(t)=1-\beta tA(b)\E S_1
+\frac{t^2}{2}\left[\beta^2A(b)^2\E S_1^2-\beta B(b)\E S_2\right]+o(t^2).
\end{equation}
For real $\delta$ the zero process is reflection symmetric, hence $\E S_1=0$.
\end{lemma}
\begin{proof}
Put $a:=m\beta=2\delta$ and $D:=d\beta=(m-1)\beta$. Then $a>1$ and $D<a$. Choose exponents
\begin{equation}\label{eq:pq-choice}
\max\{2,D\}<p<a+1,
\qquad
\max\{1,D\}<q<a,
\qquad
2q\left(1-\frac1p\right)>1.
\end{equation}
Such a choice is always possible. Indeed, as $p\uparrow a+1$ and $q\uparrow a$, the left side of the last inequality tends to $2a^2/(a+1)>1$ because $a>1$.

Set
\[
X:=\xi'(0),\qquad Y:=\xi''(0),
\qquad M_R:=\sup_{|z|\le R}|\xi'''(z)|.
\]
By Proposition~\ref{prop:slope}, $X\in L^p$. By Lemma~\ref{lem:analytic-moments}, $Y,M_R\in L^q$ for every fixed $R$. Taylor's theorem for the entire function $\xi$ gives, for $b$ in a fixed compact set $K$ and $|t|$ sufficiently small,
\begin{equation}\label{eq:xi-random-taylor}
\xi(tb_j)-1
=t b_jX+\frac{t^2b_j^2}{2}Y+R_j(t,b),
\qquad
|R_j(t,b)|\le C_K|t|^3M_R,
\end{equation}
uniformly in $1\le j\le d$ and $b\in K$.

Apply Lemma~\ref{lem:expectation-taylor} in Appendix~\ref{app:taylor} to
\[
F(u_1,\ldots,u_d)=\prod_{j=1}^d|1+u_j|^\beta
\]
and to the expansion \eqref{eq:xi-random-taylor}.  At the origin,
\[
\partial_jF(0)=\beta,
\qquad
\partial_{jj}F(0)=\beta(\beta-1),
\qquad
\partial_{jk}F(0)=\beta^2\quad(j\ne k).
\]
The appendix lemma therefore yields, uniformly for $b\in K$,
\begin{align}\label{eq:joint-before-S}
H_b(t)
&=1+\beta tA(b)\E X\\
&\quad+\frac{t^2}{2}\left\{\beta B(b)\E Y+
\left[\beta(\beta-1)B(b)+\beta^2(A(b)^2-B(b))\right]\E X^2\right\}
+o(t^2).\notag
\end{align}
Using $X=-S_1$ and $Y=S_1^2-S_2$ from \eqref{eq:derivatives}, the coefficient of $\E S_1^2$ simplifies to $\beta^2A(b)^2$, and \eqref{eq:joint-expansion} follows. Reflection symmetry gives $\E S_1=0$.
\end{proof}

\begin{remark}\label{rem:no-palm}
The proof of Lemma~\ref{lem:joint} uses neither a Campbell--Mecke estimate nor a moment bound under reduced Palm measures. This is deliberate: Theorem 2.11 of \cite{AssiotisNajnudel2026} controls joint powers only up to total exponent $2\delta$, and the direct exponential-truncation argument would require more quantitative tail room near $2\delta=1$. The derivative-moment route above uses exactly the available ranges $p<2\delta+1$ for the trace slope and $q<2\delta$ for higher analytic derivatives.
\end{remark}

\section{Proof of the fusion theorem}\label{sec:proof}
\begin{proof}[Proof of Theorem~\ref{thm:main}]
In \eqref{eq:AN}, set
\[
b_j=a_{j+1}-a_1,\qquad 1\le j\le m-1.
\]
Lemma~\ref{lem:joint}, reflection symmetry $\E S_1=0$, and Proposition~\ref{prop:moments} give
\begin{align*}
\E\prod_{j=2}^m\bigl|\xi^{\beta,m\beta/2}(\varepsilon(a_j-a_1))\bigr|^\beta
&=1+\frac{\beta^2\varepsilon^2}{2}\E S_1^2
\left[\left(\sum_{j=1}^{m-1}b_j\right)^2-m\sum_{j=1}^{m-1}b_j^2\right]
+o(\varepsilon^2).
\end{align*}
The elementary identity
\begin{equation}\label{eq:shape-id}
m\sum_{j=1}^{m-1}(a_{j+1}-a_1)^2
-\left(\sum_{j=1}^{m-1}(a_{j+1}-a_1)\right)^2
=\sum_{1\le i<j\le m}(a_i-a_j)^2=V(a)
\end{equation}
therefore yields
\[
\E\prod_{j=2}^m\bigl|\xi^{\beta,m\beta/2}(\varepsilon(a_j-a_1))\bigr|^\beta
=1-\frac{\beta^2}{8(m\beta-1)(2m+1)}V(a)\varepsilon^2+o(\varepsilon^2).
\]
The compact-uniform assertion in Lemma~\ref{lem:joint} gives local uniformity in the collision profile $a$. Multiplying by the exact Vandermonde factor in \eqref{eq:AN} proves \eqref{eq:main}.
\end{proof}

\begin{proof}[Proof of Corollary~\ref{cor:pair}]
Set $m=2$, $a_1=0$, $a_2=1$ in Theorem~\ref{thm:main}. Then $V(a)=1$ and
\[
8(m\beta-1)(2m+1)=40(2\beta-1),
\]
which gives \eqref{eq:pair}.
\end{proof}

\begin{proof}[Proof of Theorem~\ref{thm:inverse}]
This is exactly \eqref{eq:S2value} in Proposition~\ref{prop:moments}.
\end{proof}

\section{Checks at the classical values}\label{sec:checks}
The explicit $\beta=2$ and $\beta=4$ pair correlations provide independent checks of Corollary~\ref{cor:pair}. We use the formulas recorded in \cite{QuValko2025}.

For $\beta=2$,
\[
\rho^{(2)}_2(0,x)=\frac1{4\pi^2}\left[1-\left(\frac{\sin(x/2)}{x/2}\right)^2\right].
\]
Since
\[
1-\left(\frac{\sin(x/2)}{x/2}\right)^2
=\frac{x^2}{12}\left(1-\frac{x^2}{30}+O(x^4)\right),
\]
the normalized second-order coefficient is $-1/30$. Formula~\eqref{eq:pair} gives
\[
-\frac{2^2}{40(4-1)}=-\frac1{30}.
\]

For $\beta=4$, the classical expression is
\[
\rho^{(2)}_4(0,x)=\frac1{4\pi^2}\left[1-\operatorname{sinc}^2(x)+\operatorname{sinc}'(x)\int_0^x\operatorname{sinc}(t)\,dt\right],
\qquad \operatorname{sinc}(x)=\frac{\sin x}{x}.
\]
A Taylor expansion gives
\[
1-\operatorname{sinc}^2(x)+\operatorname{sinc}'(x)\int_0^x\operatorname{sinc}(t)\,dt
=\frac{x^4}{135}\left(1-\frac{2}{35}x^2+O(x^4)\right).
\]
Our formula gives
\[
-\frac{4^2}{40(8-1)}=-\frac2{35}.
\]
Thus both classical pair cases agree exactly with Corollary~\ref{cor:pair}.

There is also a useful check of the full $m$-point shape dependence. For $\beta=2$,
\[
\rho^{(m)}_2(x_1,\ldots,x_m)
=\det\left[\frac{\sin((x_i-x_j)/2)}{\pi(x_i-x_j)}\right]_{i,j=1}^m,
\]
with the diagonal interpreted as $1/(2\pi)$. Direct expansion for the collision profiles $(0,1,2)$ and $(0,1,2,3)$ gives
\begin{align*}
\rho^{(3)}_2(0,\varepsilon,2\varepsilon)
&=\frac{\varepsilon^6}{17280\pi^3}\left(1-\frac3{35}\varepsilon^2+O(\varepsilon^4)\right),\\
\rho^{(4)}_2(0,\varepsilon,2\varepsilon,3\varepsilon)
&=\frac{\varepsilon^{12}}{96768000\pi^4}\left(1-\frac{10}{63}\varepsilon^2+O(\varepsilon^4)\right).
\end{align*}
For the two profiles, $V(a)=6$ and $V(a)=20$, respectively. The coefficient in Theorem~\ref{thm:main} is therefore $-3/35$ for $m=3$ and $-10/63$ for $m=4$, exactly matching the determinant expansion. These checks test the full shape formula, not only the pair specialization.

\section{The inverse-square threshold and a conjectural phase transition}\label{sec:threshold}
The pole in \eqref{eq:inverse} suggests a structural threshold. The natural local model, consistent with confluent Palm repulsion, is that the one-point intensity of the fused environment behaves near the origin like
\[
\rho^{(1)}_{\mathrm{fused}}(x)\asymp |x|^{m\beta},
\]
then
\[
\E\sum_x x^{-2}
\quad\text{has local model}\quad
\int_0^1x^{m\beta-2}\,dx.
\]
This model is finite exactly when $m\beta>1$, logarithmically divergent at $m\beta=1$, and power divergent when $m\beta<1$. We do not use this local-intensity statement in the proof above; proving the critical and subcritical versions with uniform constants is part of Conjecture~\ref{conj:transition}.

This is the origin of Conjecture~\ref{conj:transition}. A formal split into the scales $|x|\lesssim|\varepsilon|$, $|\varepsilon|\ll|x|\ll1$, and $|x|\gtrsim1$ suggests that the supercritical Taylor term $\varepsilon^2$ is replaced at criticality by $\varepsilon^2\log(1/|\varepsilon|)$, while below the threshold a nearby environmental particle contributes at order $|\varepsilon|^{m\beta+1}$. Establishing the critical and subcritical constants requires uniform control of reduced-Palm correlation functions at the collision scale and is not used anywhere in the proof of Theorem~\ref{thm:main}.

\begin{remark}[Why the threshold is natural]
The leading factor $\prod_{i<j}|x_i-x_j|^\beta$ is the usual log-gas repulsion and exists for every $\beta>0$. The predicted transition concerns only the next term and corresponds to the point at which the fused environment is expected to cease to have a finite inverse-square response. Thus the predicted transition does not contradict the smooth leading fusion law proved in \cite{AssiotisNajnudel2026}.
\end{remark}

\section{Further directions}
We record three problems suggested by the proof.
\begin{enumerate}[label=(\roman*)]
\item \emph{Critical and subcritical fusion.} Prove Conjecture~\ref{conj:transition}, including the exact coefficients. The expected logarithmic law at $m\beta=1$ should arise as the finite part of the pole in \eqref{eq:inverse}.
\item \emph{Higher fusion coefficients.} For $m\beta$ sufficiently large, higher inverse moments of the Hua--Pickrell environment become integrable. The same combination of Ward identities and stochastic-zeta Taylor coefficients should produce higher-order terms, with new integrability thresholds.
\item \emph{Universality beyond the exact $\Sine_\beta$ limit.} It would be natural to show that the second-order coefficient in Theorem~\ref{thm:main} is the universal microscopic fusion correction for general one-dimensional beta log-gases after bulk unfolding. Such a result would require quantitative universality at the level of fused Palm measures, rather than only ordinary correlation convergence.
\end{enumerate}

\appendix

\section{A quantitative expectation--Taylor lemma}\label{app:taylor}
The elementary lemma below is used to pass from the pathwise Taylor expansion of the stochastic-zeta entire function to a second-order expansion after expectation without assuming arbitrarily high moments.

\begin{lemma}[Expectation--Taylor lemma]\label{lem:expectation-taylor}
Let $d\ge1$, $\gamma>0$, and put $D=d\gamma$. Let $K\subset\R^d$ be compact. Suppose that real random variables $X,Y,M$ satisfy, for some $p,q$,
\begin{equation}\label{eq:app-pq}
p>\max\{2,D\},\qquad q>\max\{1,D\},\qquad
2q\left(1-\frac1p\right)>1,
\end{equation}
and $X\in L^p$, $Y,M\in L^q$. Suppose, uniformly for $b\in K$ and $1\le j\le d$,
\begin{equation}\label{eq:app-U}
U_j(t,b)=tb_jX+\frac{t^2b_j^2}{2}Y+R_j(t,b),
\qquad |R_j(t,b)|\le C_K|t|^3M.
\end{equation}
Then, with
\[
F(u):=\prod_{j=1}^d|1+u_j|^\gamma,
\qquad A(b):=\sum_jb_j,
\qquad B(b):=\sum_jb_j^2,
\]
one has, uniformly for $b\in K$,
\begin{align}\label{eq:app-expansion}
\E F(U(t,b))
&=1+\gamma tA(b)\E X\\
&\quad+\frac{t^2}{2}\left\{\gamma B(b)\E Y+
\left[\gamma(\gamma-1)B(b)+\gamma^2(A(b)^2-B(b))\right]\E X^2\right\}
+o(t^2).\notag
\end{align}
\end{lemma}
\begin{proof}
Only the tail bookkeeping requires comment.  Near the origin $F$ is $C^2$ and
\[
F(u)=F(0)+DF(0)u+\frac12D^2F(0)[u,u]+\omega(\|u\|)\|u\|^2,
\qquad \omega(r)\downarrow0.
\]
Globally, for a constant depending only on $d$ and $\gamma$,
\begin{equation}\label{eq:F-growth}
F(u)\le C(1+\|u\|^D).
\end{equation}
Choose $0<\theta<1$ sufficiently small that
\begin{equation}\label{eq:theta}
p(1-\theta)>2,\qquad q(2-\theta)>2,
\qquad 1+q(2-\theta)(1-1/p)>2.
\end{equation}
This is possible by \eqref{eq:app-pq}. Put $\varepsilon=|t|^\theta$ and, with constants enlarged uniformly over $b\in K$,
\[
W_1=C_K|tX|,\qquad W_2=C_Kt^2|Y|,\qquad W_3=C_K|t|^3M,
\qquad G_t:=\{W_1,W_2,W_3\le\varepsilon\}.
\]
On $G_t$, \eqref{eq:app-U} implies $\|U(t,b)\|\le3\varepsilon$. Moreover
\[
\E[W_1^2;G_t]=O(t^2),
\]
and, when $q<2$,
\[
\frac1{t^2}\E[W_2^2;G_t]
\le C t^2\left(\frac{\varepsilon}{t^2}\right)^{2-q}\E|Y|^q=o(1),
\]
while the cases $q\ge2$ and the $W_3$ term are easier. Hence the Taylor remainder on $G_t$ is $o(t^2)$ uniformly on $K$.

It remains to replace the quadratic expression in $U$ by the displayed terms in \eqref{eq:app-expansion}. Terms containing $Y^2$ or $M^2$ are $o(t^2)$ by the preceding truncated estimates. For the only potentially delicate mixed term, let $q'=q/(q-1)$. If $q'\le p$, then $XY\in L^1$ by H\"older. If $q'>p$, truncation by $G_t$ and H\"older give
\[
|t|\E[|XY|;G_t]
\le C|t|\left(\frac{\varepsilon}{|t|}\right)^{1-p/q'}
= C|t|^{p/q'+\theta(1-p/q')}=o(1).
\]
Thus every mixed contribution beyond the terms displayed in \eqref{eq:app-expansion} is $o(t^2)$.

We finally remove $G_t$. Markov's inequality and \eqref{eq:theta} give
\begin{align*}
\mathbb P(W_1>\varepsilon)&\le C|t|^p\varepsilon^{-p}=o(t^2),\\
\mathbb P(W_2>\varepsilon)&\le Ct^{2q}\varepsilon^{-q}=o(t^2),\\
\mathbb P(W_3>\varepsilon)&\le C|t|^{3q}\varepsilon^{-q}=o(t^2).
\end{align*}
Partition $G_t^c$ into three disjoint events $E_k$ according to which $W_k$ is largest (ties are assigned in any fixed way). On $E_k$ one has $\|U\|\le3W_k$. Using \eqref{eq:F-growth}, $p>D$ and $q>D$ therefore reduces the polynomial-growth tail to
\begin{align*}
\E[W_1^D;W_1>\varepsilon]&\le C|t|^p\varepsilon^{D-p}=o(t^2),\\
\E[W_2^D;W_2>\varepsilon]&\le Ct^{2q}\varepsilon^{D-q}=o(t^2),\\
\E[W_3^D;W_3>\varepsilon]&\le C|t|^{3q}\varepsilon^{D-q}=o(t^2).
\end{align*}
Hence both the constant and polynomial-growth parts of $F(U)$ on $G_t^c$ are $o(t^2)$. For the linear Taylor term, the same partition gives, on $\{W_1>\varepsilon\}$,
\[
|t|\E[|X|;W_1>\varepsilon]
\le C|t|^p\varepsilon^{1-p}=o(t^2),
\]
while on $\{W_2>\varepsilon\}$ H\"older gives
\[
|t|\E[|X|;W_2>\varepsilon]
\le C|t|^{1+q(2-\theta)(1-1/p)}=o(t^2)
\]
by \eqref{eq:theta}; the $W_3$ event is smaller. Finally, the terms $t^2X^2$ and $t^2Y$ restricted to $G_t^c$ are $o(t^2)$ by absolute continuity of the integrals, because $X^2,Y\in L^1$ and $\mathbb P(G_t^c)\to0$. Combining the good- and bad-event estimates proves \eqref{eq:app-expansion}, uniformly for $b\in K$.
\end{proof}

\section*{AI-assisted tools}
OpenAI ChatGPT (GPT-5.6 Sol) was used during exploratory and manuscript-preparation stages for literature discovery, proposing candidate derivations, symbolic consistency checks, stress-testing intermediate arguments, and assistance with organization and drafting. AI outputs were treated as unverified suggestions rather than mathematical authority. The author checked the mathematical claims, citations, and retained derivations against the cited literature and/or by direct calculation as applicable, and takes full responsibility for the content of the manuscript.


\small
\begin{thebibliography}{11}
\bibitem{AGS2022}
T. Assiotis, M. A. Gunes, and A. Soor,
\emph{Convergence and an explicit formula for the joint moments of the circular Jacobi $\beta$-ensemble characteristic polynomial},
Math. Phys. Anal. Geom. \textbf{25} (2022), Art. 15.
\href{https://doi.org/10.1007/s11040-022-09427-4}{doi:10.1007/s11040-022-09427-4}.

\bibitem{AssiotisNajnudel2026}
T. Assiotis and J. Najnudel,
\emph{Moments of C$\beta$E field partition function, $\Sine_\beta$ correlations and stochastic zeta},
arXiv:2602.08739, 2026.

\bibitem{BEY2014}
P. Bourgade, L. Erd\H{o}s, and H.-T. Yau,
\emph{Universality of general $\beta$-ensembles},
Duke Math. J. \textbf{163} (2014), 1127--1190.

\bibitem{DHLM2021}
D. Dereudre, A. Hardy, T. Lebl\'e, and M. Ma\"ida,
\emph{DLR equations and rigidity for the Sine-beta process},
Comm. Pure Appl. Math. \textbf{74} (2021), 172--222.

\bibitem{Forrester2010}
P. J. Forrester,
\emph{Log-Gases and Random Matrices},
London Mathematical Society Monographs Series, vol. 34, Princeton University Press, 2010.

\bibitem{LiValko2022}
Y. Li and B. Valk\'o,
\emph{Operator level limit of the circular Jacobi $\beta$-ensemble},
Random Matrices Theory Appl. \textbf{11} (2022), no. 4, 2250043.
\href{https://doi.org/10.1142/S2010326322500435}{doi:10.1142/S2010326322500435}.


\bibitem{QiuEtAl2026}
S. Qiu, Y. Qu, L. Yuan, B. Valk\'o, and S. Venancio,
\emph{The pair correlation function of the $\Sine_6$ process},
arXiv:2607.26223, 2026.

\bibitem{QuValko2025}
Y. Qu and B. Valk\'o,
\emph{On the pair correlation function of the $\Sine_\beta$ process},
arXiv:2509.15446, 2025.

\bibitem{VV2009}
B. Valk\'o and B. Vir\'ag,
\emph{Continuum limits of random matrices and the Brownian carousel},
Invent. Math. \textbf{177} (2009), 463--508.
\href{https://doi.org/10.1007/s00222-009-0180-z}{doi:10.1007/s00222-009-0180-z}.

\bibitem{VV2017}
B. Valk\'o and B. Vir\'ag,
\emph{The $\Sine_\beta$ operator},
Invent. Math. \textbf{209} (2017), 275--327.
\href{https://doi.org/10.1007/s00222-016-0709-x}{doi:10.1007/s00222-016-0709-x}.

\bibitem{ValkoVirag2023}
B. Valk\'o and B. Vir\'ag,
\emph{Palm measures for Dirac operators and the $\Sine_\beta$ process},
Stochastic Process. Appl. \textbf{163} (2023), 106--135.
\href{https://doi.org/10.1016/j.spa.2023.05.011}{doi:10.1016/j.spa.2023.05.011}.
\end{thebibliography}
\end{document}